\documentclass{article}
\usepackage[T2A]{fontenc}
\usepackage[cp1251]{inputenc}
\usepackage[russian,english]{babel}
\usepackage[tbtags]{amsmath}
\usepackage{amsfonts,amssymb,mathrsfs,amscd,msb-a,amsmath}
\JournalName{}
\numberwithin{equation}{section}
\theoremstyle{plain}
\newtheorem{conjecture}{Conjecture}
\newtheorem{theorem}{Theorem}
\theoremstyle{plain}
\newtheorem{prop}{Proposition}
\newtheorem{lemma}{Lemma}

\theoremstyle{definition}
\newtheorem{proof}{Proof}
\newtheorem{definition}{Definition}

\newtheorem{remark}{Remark}

\begin{document}

\title{Orthorecursive expansion of unity}
\author{A.\,B.~Kalmynin}
\address{International Laboratory of Mirror Symmetry and Automorphic Forms, Department of Mathematics, Higher School of Economics, Moscow, Usacheva str. 6}
\email{alkalb1995cd@mail.ru}
\author{P.\,R.~Kosenko}
\address{Department of Mathematics, University of Toronto, 40 St. George St., Toronto, ON, Canada}
\email{petr.kosenko@mail.utoronto.ca}
\date{}
\udk{}
\maketitle
\begin{abstract}\textbf{Abstract.} We study the properties of a sequence $c_n$ defined by the recursive relation

\[
\frac{c_0}{n+1}+\frac{c_1}{n+2}+\ldots+\frac{c_n}{2n+1}=0
\]

for $n \geq 1$ and $c_0=1$. This sequence also has an alternative definition in terms of certain norm minimization in the space $L^2([0,1])$. We prove estimates on growth order of $c_n$ and the sequence of its partial sums, infinite series identities, connecting $c_n$ with harmonic numbers $H_n$ and also formulate some conjectures based on numerical computations.
\end{abstract}
\begin{fulltext}
\section{Introduction}
Consider the Hilbert space $\mathcal H=L^2([0,1],dx)$, where $dx$ is a standard Lebesgue measure. It is a well-known fact that the sequence $M=\{1,x,x^2,x^3,\ldots\}$ does not form a Schauder basis of $\mathcal H$, while its linear span is dense in $\mathcal H$. So, it is rather natural to ask how to express this fact quantitatively. Let us pick the function $f\in L^2([0,1])$ and consider the sequence

$$d_n(f,M)=\inf_{(a_0,a_1,\ldots,a_{n-1})\in \mathbb C^n}||f-a_0-a_1x-\ldots-a_{n-1}x^{n-1}||_{\mathcal H}.$$

One can easily show that for arbitrary function $f$ we have $\lim\limits_{n\to +\infty}d_n(f,M)=0$, so this quantity is not useful for our purposes, as it cannot even distinguish between $M$ and orthogonal bases of $\mathcal H$. In order to resolve this problem, we came up with the following definition:

\begin{definition}
Let $v_1,v_2,v_3,\ldots$ be the sequence of vectors of a Hilbert space $H$. Suppose that $v\in H$. Define the sequence $(w_n)_{n \ge 0}$ by the following conditions:
\begin{enumerate}
\item $w_0=v$.
\item There are complex numbers $\lambda_n$ such that $w_{n+1}=w_n+\lambda_{n+1}v_{n+1}$ for all $n\geq 0$.
\item The norm $||w_{n+1}||$ is minimal among all vectors of the form $w_n+\lambda v_n$.

Then $v-w_n$ is called the \emph{orthorecursive expansion} of $v$ with respect to the system $v_1,v_2,\ldots$ and the numbers $\lambda_n$ are called the coefficients of this expansion.

\end{enumerate}
\end{definition}
One can easily show that the sequence $w_n$ is uniquely determined by the conditions 1,2 and 3, so $\lambda_n$ are also uniquely determined if $v_n\neq 0$ for all $n$. It is also very easy to prove that if $\{v_n\}$ is an orthogonal basis of $H$, then the orthorecursive expansion of any vector $v$ with respect to $\{v_n\}$ converges to $v$ and orthorecursive expansion of $v_1$ with respect to the sequence $\{v_{n+1}\}$ is just the sequence of zero vectors. Neither of this holds for our sequence $M\subset \mathcal H$, so orthorecursive expansions at least allow us to distinguish between $M$ and orthogonal bases. In this paper, we will study one very particular orthorecursive expansion. More precisely, the sequence $c_n$ mentioned in the abstract is the sequence of coefficients of orthorecursive expansion of $1$ with respect to the sequence $\{x,x^2,x^3,\ldots\}$.

\begin{definition}
The sequence $c_n$ for $n\geq 1$ is the sequence of coefficients of orthorecursive expansion of the function $1$ with respect to the system $\{x,x^2,\ldots\}$.
\end{definition}

It is also convenient to define $c_0=1$ and $p_n(x)=c_0+c_1x+\ldots+c_nx^n$.

Later we will discuss the proof of the following fact:

\begin{prop}
\label{recurrent}
The sequence $c_n$ satisfies the relation

\begin{equation}
    \frac{c_0}{n+1}+\ldots+\frac{c_n}{2n+1}=\sum_{0\leq k\leq n}\frac{c_k}{n+1+k}=\delta_{n0}
\end{equation}

for all $n\geq 0$.
\end{prop}

This proposition is rather useful for our numerical considerations, as it allows us to compute $c_n$ very fast. For example, one can compute the first few terms of our sequence. It starts as follows:

\[
1,-3/2,5/24,77/720,277/4480,140173/3628800\ldots
\]

Extended computation of $c_n$ reveals rather peculiar properties of this sequence. Unfortunately, it also turns out that it is difficult to control the behavior of $c_n$. In subsequent sections we will prove some results on the growth of $c_n$ and the sequence of partial sums of our sequence, deduce the determinant formula for $c_n$ and the integral equation for its ordinary generating function, formulate certain conjectures on $c_n$ that are based on numerical evidence and finally find certain family of infinite series that connect $c_n$ with harmonic numbers and $\pi$. 

\begin{remark}
This construction has been studied by several authors before, even for an abstract Hilbert space, for example, see \cite{luk1},\cite{luk2}. In contrast to the mentioned papers, our work is more concentrated around the study of some particular sequence of rational numbers that arises as a sequence of coefficients of a particular orthorecursive expansion. Therefore, our paper can be considered as more number-theoretic than functional-analytic, and indeed, some rather subtle number-theoretic properties of permutations provide us with the proof that $c_n$ is never zero (see Section 4). We also believe that our sequence is related to several other topics in number theory, such as the properties of Dirichlet series (Section 5) and harmonic numbers (Section 6).

\end{remark}


\section{Basic properties and alternative definitions of \texorpdfstring{$c_n$}{cn}.}

In this section we will prove the Proposition \ref{recurrent} and deduce two new ways to define $c_n$, using determinants of almost lower diagonal matrices and solution of certain integral equation.

Now we start the

 \begin{proof}[of Proposition \ref{recurrent}]
 Let us define the polynomials $p_n(x)$ by $p_n(x)=c_0+\ldots+c_nx^n$, as in previous section. By the main definition we see that for any $n \geq 0$ the norm of $p_n(x)+cx^{n+1}$ is minimal when $c=c_{n+1}$. From this it is easy to see that $c_{n+1}$ is always real. 
 
 Indeed, recall that $c_0  = 1$, and if $c_k\in \mathbb R$ for all $0\leq k\leq n$ then for all complex $c$ we have $||p_n(x)+cx^{n+1}||\geq ||p_n(x)+(\mathrm{Re}\,c)x^{n+1}||$ and equality is attained only for $c \in \mathbb R$, so $c_{n+1} \in \mathbb R$, as needed.
 
 Knowing that, we notice that if the norm $\left\langle p_{n-1} + c x^n, p_{n-1} + c x^n \right\rangle$ attains the minimum at $c \in \mathbb{R}$, then it is the critical point of the inner product as a function of $c$:
 \begin{equation}
	\frac{\partial}{\partial c_n} \left\| p_n \right\|_2^2 = \frac{\partial}{\partial c_n} \left\langle p_{n-1} + c_n x^n, p_{n-1} + c_n x^n  \right\rangle =  2 \left\langle p_{n-1}, x^n \right\rangle + 2 \left\langle x^n, x^n \right\rangle c_n = 0.
	\end{equation}
	This immediately implies that
	\begin{equation}
	    c_n = -\dfrac{\left\langle p_{n-1}, x^n \right\rangle }{ \left\langle x^n, x^n  \right\rangle } = -(2n+1) \left\langle p_{n-1}, x^n \right\rangle
	\end{equation}
	and
	\begin{equation}
	\label{prop1}
	    \left\langle p_n, x^n \right\rangle = 0
	\end{equation}
	for any $n > 0$. Now let us expand the inner product in \eqref{prop1}:
	\begin{equation}
	    \left\langle p_n, x^n \right\rangle = \left\langle \sum_{k=0}^n c_k x^k, x^n \right\rangle = \sum_{k=0}^n \dfrac{c_k}{n+k+1} = 0.
	\end{equation}
	The relation \eqref{prop1} has a simple geometric interpretation: as $p_n(x)$ is the shortest possible vector that connects the point $p_{n-1}(x)$ with the line $\{\lambda x^n\}_{\lambda\in \mathbb R}$, it should be orthogonal to this line, which is indeed the case.
 \end{proof} 
\subsection{Expressing the coefficients as determinants}

Let's write down the first few instances of \eqref{recurrent}:
	\[
	\begin{split}
	\frac{c_0}{2} + \frac{c_1}{3} &= 0 \\
	\frac{c_0}{3} + \frac{c_1}{4} + \frac{c_2}{5} &= 0 \\
	\frac{c_0}{4} + \frac{c_1}{5} + \frac{c_2}{6} + \frac{c_3}{7} &= 0 \\
	& \dots
	\end{split}
	\]
	Notice that for each $n$ the first $n$ equations form a system of linear equations:
	\[
	\begin{split}
	\frac{c_1}{3} &= -\frac{1}{2} \\
	\frac{c_1}{4} + \frac{c_2}{5} &= -\frac{1}{3} \\
	\frac{c_1}{5} + \frac{c_2}{6} + \frac{c_3}{7} &= -\frac{1}{4} \\
	& \dots \\
	\frac{c_1}{n+2} + \dots + \frac{c_n}{2n+1} &= -\frac{1}{n+1}.
	\end{split}
	\]
	Therefore, due to the Kramer's rule, we have the following expression for $c_n$:
	\begin{equation}
	\label{determinant}
	\begin{aligned}
		c_n &= (2n+1)!! \begin{vmatrix}
	\frac{1}{3} & 0 & 0 & \dots & -\frac{1}{2} \\[6pt]
	\frac{1}{4} & \frac{1}{5} & 0 & \dots & -\frac{1}{3} \\[6pt]
	\frac{1}{5} & \frac{1}{6} & \frac{1}{7} & \dots & -\frac{1}{4} \\[6pt]
	& &  & \ddots & \\[6pt] 
	\frac{1}{n+2} & \frac{1}{n+3} & \frac{1}{n+4} & \dots & -\frac{1}{n+1}
	\end{vmatrix} = \\ &= (-1)^n (2n+1)!! \begin{vmatrix}
	\frac{1}{2} & \frac{1}{3} & 0 & \dots & 0 \\[6pt]
	\frac{1}{3} & \frac{1}{4} & \frac{1}{5} & \dots & 0 \\[6pt]
	\frac{1}{4} & \frac{1}{5} & \frac{1}{6} & \dots & 0 \\[6pt]
	& &  & \ddots & \\[6pt] 
	\frac{1}{n+1} & \frac{1}{n+2} & \frac{1}{n+3} & \dots & \frac{1}{2n}
	\end{vmatrix} := \\ &:= (-1)^n(2n+1)!! \det(A_n),
	\end{aligned}
	\end{equation}
	where
	\[
	(A_n)_{ij} = \begin{cases}
	\dfrac{1}{j+i},& j - i \le 1, \\
	0, & \text{ otherwise}
	\end{cases}
	\]
	for all $n \in \mathbb{N}$, and $1 \le i, j \le n$.
	\subsection{An integral equation for the generating function of \texorpdfstring{$c_n$}{cn}}
	
 Using the identity \eqref{prop1}, one can derive an integral equation for the ordinary generating function of $c_n$. Let $F(t)=c_0+c_1t+c_2t^2+\ldots=\sum\limits_{k\geq 0}c_kt^k$. In subsequent sections we will prove that this series converges for all $|t|\leq 1$. Here we will prove that for any $0\leq t<1$ we have
 \begin{equation}
 \label{integraleq}
 \int_0^1 \frac{F(xt^2)}{1-tx}\mathop{dx}=1.
 \end{equation}
 To prove this formula, let us multiply the identities \eqref{prop1} by $t^n$ and sum over all $n\geq 0$. Then we get
 
 \[
 \sum_{n\geq 0}t^n\langle x^n,p_n(x)\rangle=1+0t+0t^2+\ldots=1
 \]
 
 Next, note that for every $n$ the $n$-th summand on the left-hand side is equal to
 
 \[
 t^n\langle x^n,p_n(x)\rangle=t^n\langle x^n,\sum_{k\leq n}c_kx^k \rangle=
 \sum_{k\leq n} c_k\langle t^n x^n, x^k\rangle=\sum_{k \leq n}\langle c_k t^n x^{n+k}, 1\rangle.
 \]
 
 Therefore,
 
 \[
 \sum_{k\geq 0} \langle c_k x^k \sum_{n\geq k} t^n x^n,1 \rangle=\sum_{n\geq k\geq 0}\langle c_k t^n x^{n+k}, 1\rangle=1.
 \]
 
 Next, for all nonnegative $k$ we have
 
 \[
 \sum_{n\geq k}t^nx^n=(tx)^k\sum_{n\geq 0}(tx)^n=\frac{(tx)^k}{1-tx}.
 \]
 
 Using this identity, we conclude that
 
 \[
 \int_0^1 \frac{F(xt^2)}{1-tx}dx=\left\langle \frac{F(tx^2)}{1-tx},1 \right\rangle=\left\langle\sum_{k\geq 0}\frac{c_k(tx^2)^k}{1-tx},1\right\rangle=\left\langle\sum_{k\geq 0}c_k x^k\sum_{n\geq k}t^nx^n,1 \right\rangle=1,
 \]
 
 which proves the desired formula.
 
 Unfortunately, we were not able to find a solution of this integral equation in any form other than $F(t)=\sum\limits_{k\geq 0} c_kt^k$.
 
\section{Upper bounds and partial sums}
 
 In order to get a better understanding of the sequence $(c_n)$, it is natural to ask what is its growth order. Using the formula \eqref{prop1}, one can easily show that
 
 \[
 \begin{aligned}
  ||p_n(x)||^2 &=||p_{n+1}(x)-c_{n+1}x^{n+1}||^2= \\ &=||p_{n+1}(x)||^2-2c_{n+1}\langle p_{n+1}(x),x^{n+1}\rangle+\frac{c_{n+1}^2}{2n+3}=||p_{n+1}(x)||^2+\frac{c_{n+1}^2}{2n+3}
 \end{aligned}
 \]
 
for all $n\geq 0$, so that for all $n$
\begin{equation}
\label{normp}
||p_n||^2=1-\sum_{1\leq k\leq n}\frac{c_k^2}{2k+1}.
\end{equation}
Therefore, due to positivity of the norm, we prove that $c_n=O(\sqrt{n})$, that is, $c_n$ grows at most polynomially. In fact, it turns out that $c_n$ \begin{it}{decreases}\end{it} with an at least polynomial rate. The goal of this section is to prove the following result:

\begin{theorem}
\label{32estimate}
There is a positive constant $C$ such that for all $n\geq 1$ we have
\begin{equation}
\label{bestestim}
|c_n|\leq \frac{C}{n^{3/2}}    
\end{equation}
\end{theorem}

To prove this theorem, we are going to introduce two auxiliary quantities. For all $n\geq 0$ we will denote by $s_n=c_0+c_1+\ldots+c_n$ the sequence of partial sums of $c_n$ and we also set

\[
D(n)=\int_0^1 p_n'(x)^2 \mathop{dx},
\]

i.e. $D(n)$ is a sequence of squared $L^2$-norms of the derivatives of polynomials $p_n(x)$. For example, one can easily see that $D(0)=0, D(1)=9/4$ and $s_0=1$. Numerical computations also show that, for example $s_{100}\approx 0.001888$ and $s_{729}\approx -0.000124$, so it is reasonable to conjecture that $s_n\to 0$ as $n\to \infty$. This is indeed the case, as we will demonstrate later. 

Let us establish some relations between $c_n, s_n$ and $D(n)$. In this section our main tools are Cauchy-Bunyakovsky-Schwarz inequality and integration by parts.

\begin{lemma}
\label{lemma1}
For all $n\geq 2$ the inequality

\[
|c_n|\leq \sqrt{D(n)}n^{-3/2}
\]

holds.

\end{lemma}

\begin{proof}
Let's consider the following integral:

\[\int_0^1 p_n'(x)x^n(1-x)\mathop{dx}.
\]

As the integrand is zero at the boundary points of our interval, one application of integration by parts gives

\[\int_0^1 p_n'(x)x^n(1-x)\mathop{dx}=\int_0^1 x^n(1-x) d p_n(x)=-\int_0^1 p_n(x)(nx^{n-1}-(n+1)x^n)\mathop{dx}.
\]

Now, the formula \eqref{prop1} allows us to get rid of the second summand, as $p_n$ and $x^n$ are orthogonal. Thus, we obtain

\begin{equation}
\label{pncn}
\begin{aligned}
\int_0^1 p_n'(x)x^n(1-x)\mathop{dx}=-n\int_0^1 p_n(x)x^{n-1}\mathop{dx}=\\
=-n\int_0^1 c_n x^{2n-1}\mathop{dx}-n\int_0^1 p_{n-1}(x)x^{n-1}\mathop{dx}=-c_n/2,
\end{aligned}
\end{equation}

due to orthogonality of $p_{n-1}(x)$ and $x^{n-1}$. 

Next, by Cauchy-Bunyakovsky-Schwarz inequality, the integral on the left side of our equality can be estimated as follows:

\[
\begin{aligned}
\left|\int_0^1 p_n'(x)x^n(1-x)\mathop{dx}\right|&=|(p_n'(x),x^n(1-x))|\leq ||p_n'(x)||\cdot||x^n(1-x)||= \\ &=\sqrt{D(n)}||x^n(1-x)||.
\end{aligned}
\]

Finally, the square of the last norm is equal to

\[\int_0^1 x^{2n}(1-x)^2\text{d}x=\frac{\Gamma(2n+1)\Gamma(3)}{\Gamma(2n+4)}=\frac{2}{(2n+3)(2n+2)(2n+1)}\leq \frac{1}{4n^3}.
\]

Therefore, we obtain

\[|c_n|/2\leq \sqrt{D(n)}\sqrt{\frac{1}{4n^3}}=\sqrt{D(n)}n^{-3/2}/2,
\]

which proves the desired inequality.
\end{proof}

Next lemma shows that the behavior of $s_n$ is also controlled by $D(n)$ in a similar manner.

\begin{lemma}
\label{lemma2}
For all $n\geq 1$ we have

\[s_n^2\leq \frac{D(n)}{2n+3}
\]
\end{lemma}
\begin{proof}
Indeed, integrating by parts we deduce

\[\int_0^1 p_n'(x)x^{n+1}\mathop{dx}=\int_0^1 x^{n+1} d p_n(x)=p_n(1)-(n+1)\int_0^1 x^n p_n(x)\mathop{dx}.
\]

Yet another application of \eqref{prop1} together with a simple observation that $p_n(1)=c_0+c_1+\ldots+c_n=s_n$ gives

\begin{equation}
\label{pnsn}
    \int_0^1 p_n'(x)x^{n+1}\mathop{dx}=s_n.
\end{equation}
Next, by Cauchy-Bunyakowski-Schwarz, we obtain

\[s_n^2=\langle p_n'(x),x^{n+1}\rangle^2\leq ||p_n'(x)||^2||x^{n+1}||^2=\frac{D(n)}{2n+3},
\]

as needed.
\end{proof}

The next formula is a core of our argument, as it allows to ``reverse'' the lemmas \ref{lemma1} and \ref{lemma2} in a certain sense, i.e. it provides us with a bound for $D(n)$ in terms of $s_n$ and $c_k$ for $0\leq k\leq n$.

\begin{lemma}
\label{lemma3}
For every $n\geq 1$

\[
D(n)\leq ns_n^2+(n+1)c_n^2/2+\sum_{k=0}^{n-1} (k+1)c_k^2.
\]

\end{lemma}
\begin{proof}
We will proceed by induction. First, for $n=1$ we have $D(1)=9/4$ on the left side and $s_1^2+c_1^2+\sum\limits_{0\leq k\leq 0}c_k^2=1/4+9/4+1=7/2>9/4$, as needed. Suppose that our inequality holds for $n=k$. For $n=k+1$ we have
\[
D(n)=D(k+1)=\langle p_{k+1}'(x),p_{k+1}'(x)\rangle=\langle p_k'+(k+1)c_{k+1}x^k,p_k'(x)+(k+1)c_{k+1}x^k\rangle.
\]

Expanding this inner product and using \eqref{pncn} and \eqref{pnsn}, we deduce

\begin{equation}
\label{pndn}
\begin{aligned}
D(k+1)=D(k)+2(k+1)c_{k+1}\langle p_k',x^k\rangle+\frac{(k+1)^2c_{k+1}^2}{2k+1}=\\
=D(k)+(k+1)c_{k+1}(2s_k-c_k)+\frac{(k+1)^2c_{k+1}^2}{2k+1}.
\end{aligned}
\end{equation}

Let us note that by AM-GM inequality we have $-c_kc_{k+1}\leq c_k^2/2+c_{k+1}^2/2$. Also, $ks_k^2+2(k+1)c_{k+1}s_k=(k+1)((s_k+c_{k+1})^2-c_{k+1}^2)-s_k^2\leq (k+1)s_{k+1}^2-(k+1)c_{k+1}^2.$ Therefore, as $\frac{(k+1)^2}{2k+1}\leq k+1$, we obtain

\[D(k+1)\leq D(k)+(k+1)s_{k+1}^2+(k+1)c_k^2/2+(k+1)c_{k+1}^2/2-ks_k^2.
\]

Next, by the inductive assumption,

\[D(k)-ks_k^2\leq (k+1)c_k^2/2+\sum_{0\leq m\leq k-1}(m+1)c_m^2.
\]

Hence, from the previous inequality we see that

\[D(k+1)\leq (k+1)s_{k+1}^2+(k+1)c_{k+1}^2/2+(k+1)c_k^2/2+(k+1)c_k^2/2+\sum_{0\leq m\leq k-1}(m+1)c_m^2\leq\]
\[\leq (k+1)s_{k+1}^2+(k+2)c_{k+1}^2/2+\sum_{0\leq m\leq k}(m+1)c_m^2,
\]

which completes the proof.
\end{proof}
\begin{remark}
From the recurrence relation \eqref{pndn} one can deduce a more precise formula for $D(k)$, but it is also more unwieldy and is not very helpful for our purposes.
\end{remark}

Lemmas \ref{lemma1},\ref{lemma2} and \ref{lemma3} allow us to prove Theorem \ref{32estimate}:

\begin{proof}[of Theorem \ref{32estimate}]
From the Lemma \ref{lemma2} we get

\[(2n+3)s_n^2\leq D(n).
\]
Lemma \ref{lemma3} then implies that
\begin{equation}
\label{partialbound}
(n+3)s_n^2\leq \frac{n+1}{2}c_n^2+\sum_{0\leq k\leq n-1}(k+1)c_k^2.
\end{equation}
Note that this inequality is far superior to the trivial AM-GM bound 
\[
s_n^2=(c_0+\ldots+c_n)^2\leq (n+1)c_0^2+(n+1)c_1^2+\ldots+(n+1)c_n^2,
\]

which suggests that the sequence $(c_n)$ oscillates enough to cause some nontrivial cancellation in partial sums.

From the Lemma \ref{lemma3} and \eqref{partialbound} we deduce
\[
D(n)\leq (n+3)s_n^2+\frac{(n+1)}{2}c_n^2+\sum_{0\leq k\leq n-1}(k+1)c_k^2\leq (n+1)c_n^2+2\sum_{0\leq k\leq n-1}(k+1)c_k^2.
\]

Next, Lemma \ref{lemma1} gives us the following estimate:

\[
n^3c_n^2\leq D(n)\leq (n+1)c_n^2+2\sum_{0\leq k\leq n-1}(k+1)c_k^2.
\]

Thus,

\[
(n^3-n-1)c_n^2\leq 2\sum_{0\leq k\leq n-1}(k+1)c_k^2.
\]

For $n\geq 2$ we have $n^3-n-1\geq 0.5n^3$, therefore for $n\geq 2$

\begin{equation}
\label{themostimportantinequality}
    c_n^2 \leq \frac{4}{n^3}\sum_{0\leq k\leq n-1}(k+1)c_k^2.
\end{equation}

We will conclude our proof by the repeated application of the inequality \eqref{themostimportantinequality}. From \eqref{normp} we deduce that $c_n^2\leq 2n+1$ for all $n$. Therefore, for all $n\geq 2$

\[
c_n^2\leq \frac{4}{n^3}\sum_{0\leq k\leq n-1}(k+1)(2k+1).
\]

All summands on the right-hand side are $\leq 2n^2$, so for all $n\geq 2$ we have $c_n^2\leq 8$. This inequality is also true for $n=0$ or $1$, as $c_0^2<c_1^2=9/4<8$. Therefore,

\begin{equation}
\label{prac}
c_n^2\leq \frac{4}{n^3}\sum_{0\leq k\leq n-1}8(k+1)\leq \frac{32}{n}.
\end{equation}

Last inequality is true because every summand is $\leq 8n$. Applying the inequality \eqref{themostimportantinequality} one more time we get for all $n\geq 2$

\[c_n^2\leq \frac{4}{n^3}\left(1+\sum_{1\leq k\leq n-1}\frac{32(k+1)}{k}\right).\]

For all $k\geq 1$ we have $\frac{k+1}{k}\leq 2$, so

\[c_n^2\leq \frac{128}{n^2}\leq \frac{512}{(n+1)^2}.\]

Consequently, we have

\[c_n^2\leq \frac{4}{n^3}\sum_{0\leq k\leq n-1}\frac{512(k+1)}{(k+1)^2}=\frac{2048H_n}{n^3},\]

where $H_n=1+1/2+\ldots+1/n$ is the $n$-th harmonic number, where $H_0 = 0$. Applying this bound we compute

\[c_n^2 \leq \frac{4}{n^3}\sum_{k=1}^{+\infty}\frac{2048(k+1)H_k}{k^3}=\frac{16384\zeta(3)+4+1024\pi^4/9}{n^3}<\frac{30782}{n^3},\]

which completes the proof.
\end{proof}

\begin{remark}
Of course, the constant $30782$ in the resulting inequality is not optimal. Computations also suggest that the exponent $3/2$ is not optimal either. Let us define
\[
\delta=\limsup_{n\to \infty}\frac{\ln |c_n|}{\ln n}.
\]

Theorem \ref{32estimate} implies that $\delta\leq -3/2$. We think that $\delta>-\infty$. A.V. Ustinov conjectured that $\delta=-7/3$ and this is supported by calculations. For example, $\ln c_{5555}\approx(-7/3+0.00017)\ln 5555$.
\end{remark}

\begin{remark}
Due to the formula \eqref{normp}, the sequence $||p_n||^2$ is nonnegative and nonincreasing. Therefore, there is a limit $K:=\lim\limits_{n\to +\infty}||p_n||^2.$
The inequality \eqref{prac} can be used to compute $K$ to arbitrary precision. It allows us to show that $K\approx 0.239037$. No other expressions for $K$ than $K=1-c_1^2/3-c_2^2/5-\ldots$ is known yet.
\end{remark}

\section{On arithmetic properties of \texorpdfstring{$c_n$}{cn}}
In this section we compute 2-adic norm of $c_n$ for every $n$ and prove that $c_n$ is always nonzero as an easy consequence.
\begin{prop}
\label{powersoftwo}
Denote $c_n = \dfrac{p_n}{q_n}$, where $q_n > 0$ and $(p_n,q_n)=1$. Then 
\[
q_n = 2^{2n - b(n)}r_n,
\]
where $r_n$ is odd, and $b(n)$ equals the number of non-zero digits in the binary expansion of $n$. In other words, the 2-adic norm of $c_n$ equals $|c_n|_2 = 2^{2n - b(n)}$. In particular, $c_n \ne 0$ for all $n \ge 0$.
\end{prop}
\begin{proof}
First of all, let us prove a simple lemma.
\begin{lemma}
\label{permut}
Let $\sigma \in S_n$ with $\sigma(i) \le i+1$. Then $\sigma$ can be decomposed as a product of long cycles as follows:
\[
\sigma = \prod_{i=1}^{m-1} (k_i \: k_i + 1 \dots k_{i+1}-1)
\]
for some $1 \le k_1 < \dots < k_m \le n$.
\end{lemma}
\begin{proof}
    Notice that if $\sigma = (a_1 \dots a_k)$ is a long cycle itself, then the statement of the lemma is fairly obvious, because then $a_{i+1} = a_i + 1$ for all $1 \le i \le k$. However, every permutation admits a decomposition as a product of pairwise nonintersecting long cycles, which satisfy the conditions of the Lemma as well. This is because if the product $\sigma_1\sigma_2$ satisfies $\sigma_1\sigma_2(i)\leq i+1$ for all $i$ and $\sigma_1$ and $\sigma_2$ have nonintersecting supports, then $\sigma_j(i)\leq i+1$ holds for $j=1,2$ and all $i$.
\end{proof}
Let us look at the determinant relation \eqref{determinant}. It implies that for every $n > 0$ we have
\begin{equation}
    \label{permsum}
    c_n = (-1)^n (2n+1)!! \sum_{\substack{\sigma \in S_n \\ \sigma(i) \le i+1}} (-1)^\sigma \dfrac{1}{1+\sigma(1)} \dots \dfrac{1}{n + \sigma(n)}.
\end{equation}
The double factorial is odd, therefore,
\[
|c_n|_2 = \left| \sum_{\substack{\sigma \in S^n \\ \sigma(i) \le i+1}} (-1)^\sigma \dfrac{1}{1+\sigma(1)} \dots \dfrac{1}{n + \sigma(n)} \right|_2.
\]

Now we apply the Lemma \ref{permut} and notice that if $(a, a+1, \dots, a+k)$ is a long subcycle of $\sigma$ for some $k > 0$, then
\[
\left| \dfrac{1}{a + a+1} \dots \dfrac{1}{a + k-1 + a+k} \cdot \dfrac{1}{2a + k} \right|_2  = \left| \dfrac{1}{2a+k} \right|_2,
\]
because the other denominators are odd. Now we claim that
\[
\left| \dfrac{1}{2a+k} \right|_2 < \left| \dfrac{1}{2a} \dots \dfrac{1}{2(a+k-1)} \cdot \dfrac{1}{2(a + k)} \right|_2 = 2^{k+1} \left| \dfrac{(a-1)!}{(a+k)!} \right|_2.
\]
where the right product corresponds to the trivial permutation. 

If $k$ is odd, then $\left| \dfrac{1}{2a+k} \right|_2 = 1$, if $k$ is even, then $(2a + k) | 2a \cdot \dots \cdot 2(a+k-1) 2(a + k)$.
Therefore, for any nontrival permutation $\sigma\in S_n$ with $\sigma(i)\leq i+1$ for all $i$ we have

\[
\left|\frac{1}{1+\sigma(1)}\ldots\frac{1}{n+\sigma(n)}\right|_2<\left|\frac{1}{2^n n!}\right|_2=2^{2n-b(n)}.
\]
In other words, the maximum is achieved only for the trivial permutation.
Now, applying the ultrametric property of 2-adic norm we get
\[
\left|c_n-\frac{1}{2^n n!}\right|_2 \leq \max_{\substack{\sigma \in S^n\backslash\{id\} \\ \sigma(i) \le i+1}} \left| \dfrac{1}{1+\sigma(1)} \dots \dfrac{1}{n + \sigma(n)} \right|_2< \left| \dfrac{1}{2} \dots \dfrac{1}{2(n-1)} \dfrac{1}{2n} \right|_2 = 2^{2n - b(n)}.
\]

Applying ultrametricity once more, we get

\[
\left|c_n\right|_2=\max\left(\left|\frac{1}{2^n n!}\right|_2,\left|c_n-\frac{1}{2^n n!}\right|_2\right)=2^{2n-b(n)},
\]

which completes the proof.
And as the $2$-adic norm of $c_n$ is nonzero for every $n$, $c_n \ne 0$. 
\end{proof}
Furthermore, from the determinant formula \eqref{determinant} it is easy to see that $\frac{c_n(2n)!}{(2n+1)!!}$ is an integer. The fact that it is nonzero yields the following estimate

\[
|c_n|\geq \frac{(2n+1)!!}{(2n)!}.
\]

This lower bound is rather weak, as it decreases super-exponentially. The reasonable conjecture is that $|c_n|$ admits a polynomially decreasing lower bound.

\section{Sign changes and Dirichet series of \texorpdfstring{$c_n$}{cn}}
Here we will formulate a conjecture regrading the sign changes of $c_n$.

\begin{definition}
Let $(a_n)_{n \ge 0}$ be a sequence of real numbers. Then we say that $(a_n)$ changes sign at $N > 0$, if $a_N a_{N+1} < 0$. 
\end{definition}

\begin{conjecture}
\label{signchanges}
The sequence $(c_n)_{n \ge 0}$ changes sign infinitely often.
\end{conjecture}

Numerical experiments show that among the first 20000 terms of $(c_n)$  sign changes occur at the following values of $n$:

\[
0,1,27,533,10457,\dots
\]
This suggests that if the second part of the Conjecture \ref{signchanges} holds, then the sign changes are exponentially rare, for example, notice that
\[
\dfrac{27}{1} = 27, \quad \frac{533}{27} = 19.7(407), \quad \frac{10457}{533} \approx 19.61.
\]
Let us denote the $(n+1)$-th change of sign of our sequence by $t_n$.
We conjecture that the ratio $\dfrac{t_{n+1}}{t_n}$ approaches some limit as $n$ goes to infinity.

So, it seems plausible that the sequence $c_n$ oscillates with a certain sort of ``logarithmic phase''. More precisely, we formulate the following:

\begin{conjecture}
\label{preciseasympt}
There are real constants $\delta, \varphi,A$ and $P$ such that for all $n>0$ we have
\begin{equation}
\label{asymp}
c_n=\frac{A}{n^\delta}\sin(P\ln n+\varphi)+O(n^{-\delta-\varepsilon}),
\end{equation}

where $\varepsilon>0$.
\end{conjecture}

\subsection{Dirichlet series of \texorpdfstring{$c_n$}{cn}}

Suppose that $C(s)$ is a Dirichlet generating function of $c_n$, that is

\[
C(s)=\sum_{n=1}^\infty \frac{c_n}{n^s}.
\]

This Dirichlet generating series is rather hard to work with: note that the whole Section 3 of current paper is devoted to the proof of (the slightly stronger form of) the following result:
\begin{theorem}
The series $C(s)$ converges absolutely for any complex $s$ with $\mathrm{Re}\,s>-1/2$. Also, $C(0)=-1$.
\end{theorem}

\begin{proof}
Indeed, by Theorem \ref{32estimate} we have $|c_n|\leq \dfrac{C}{n^{3/2}}$ for some $C>0$. Therefore, if $s=\sigma+it$ with $\sigma, t\in \mathbb R$ we have for each summand of our series

\[ \left|\frac{c_n}{n^s}\right|\leq \frac{C}{n^{\sigma+3/2}},
\]

which provides convergence for all $\sigma>-1/2$ by the comparison with the series $\sum\limits_n n^{-\sigma-3/2}$, which is standard.

To prove the second part of our theorem, note that the Lemmas \ref{lemma2} and \ref{lemma3} give us the following result:

\[(2n+3)s_n^2\leq D(n) \leq ns_n^2+(n+1)c_n^2/2+\sum_{0\leq k\leq n-1}(k+1)c_k^2.
\]

Therefore,

\[s_n^2 \leq \frac{1}{n}\left((n+1)c_n^2/2+\sum_{0\leq k\leq n-1}(k+1)c_k^2\right).
\]

Applying Theorem \ref{32estimate}, we get

\[(n+1)c_n^2/2+\sum_{0\leq k\leq n-1}(k+1)c_k^2\leq C^2\left(\frac{n+1}{2n^3}+\sum_{1\leq k\leq n-1}\frac{k+1}{k^3}+1\right)\leq C^2 C_1,
\]

where $C_1=1+\sum\limits_{k\geq 1}\frac{k+1}{k^3}=1+\zeta(2)+\zeta(3)$. Therefore, $s_n^2\leq \frac{C_2}{n}$ for some $C_2>0$. Next, we have

\[C(0)=\sum_{n=1}^{+\infty} c_n=\lim_{n\to \infty} (c_1+\ldots+c_n)=\lim_{n\to \infty}(s_n-1)=\lim_{n\to \infty}(O(n^{-1/2})-1)=-1,
\]

which completes the proof.
\end{proof}

One of the simplest, but rather unexpected consequences of Conjecture \ref{preciseasympt} is the following statement:

\begin{prop}
If Conjecture \ref{preciseasympt} holds then the function $C(s)$ has a meromorphic extension to the half-plane $\mathrm{Re}\,s>1-\delta-\varepsilon$ with only simple poles in two points $s=s_0=1-\delta+iP$ and $s=\overline{s_0}=1-\delta-iP$.
\end{prop}

\begin{proof}
Indeed, let $b_n=c_n-\frac{A}{n^\delta}\sin(P\ln n+\varphi).$ Then by Conjecture \ref{preciseasympt} we have $|b_n|=O(n^{-\delta-\varepsilon})$, hence the series

\[B(s)=\sum_{n=1}^{+\infty} \frac{b_n}{n^s}
\]

converges absolutely and uniformly in every closed half-plane of the form $\mathrm{Re}\,s\geq 1-\delta-\varepsilon+\zeta$ with $\zeta>0$. Therefore, $B(s)$ is a holomorphic function in the open half-plane $\mathrm{Re}\,s>1-\delta-\varepsilon$. Now let us note that for $C(s)$ we have the following decomposition
\[C(s)=\sum_{n=1}^{+\infty} \frac{b_n+An^{-\delta}\sin(P\ln n+\varphi)}{n^s}=B(s)+\sum_{n=1}^{+\infty} \frac{An^{-\delta}\sin(P\ln n+\varphi)}{n^s}.
\]

We are left with the second series. To handle this, let us use the Euler's formula for sine and observe that

\[\frac{n^{-\delta}\sin(P\ln n+\varphi)}{n^s}=\frac{1}{2in^{\delta+s}}(\exp(iP\ln n+i\varphi)-\exp(-iP\ln n-i\varphi))=
\]
\[\frac{e^{i\varphi}n^{iP}-e^{-i\varphi}n^{-iP}}{2in^{\delta+s}}.
\]

From this formula we finally obtain

\[
C(s)=B(s)+\frac{e^{i\varphi}A}{2i}\zeta(\delta+s-iP)-\frac{e^{-i\varphi}A}{2i}\zeta(\delta+s+iP),
\]

which completes the proof of this proposition, as $B(s)$ is holomorphic inside the region $\mathrm{Re}\,s>-1/2-\varepsilon$ and the functions $\zeta(\delta+s\pm iP)$ have a meromorphic continuation the the whole complex plane with only simple poles at $s=1-\delta\mp iP$ (for some properties of $\zeta(s)$ see \cite{ivic},\cite{kar}).
\end{proof}

This proposition shows that Conjecture \ref{preciseasympt}, if true, would be a very remarkable property, as the Dirichlet generating function of a sequence with rather simple recurrent formula is not expected to have complex poles.


\section{Infinite series identities with harmonic numbers, $c_n$ and $\pi$}
Here we prove that some infinite family of series containig $H_n$, $c_n$ and (in one summand) $\pi$, can be evaluated in a closed form.

\begin{theorem}
\label{harmonicrelation}
Suppose that $r\geq 0$, $r\in \mathbb Z$. Define the sequence $\{h_r(n)\}_{n\geq 0}$ by the following formula:
\[
h_r(n)=\begin{cases} \frac{H_{2n}-H_{n+r}}{n-r} \text{ if } n\neq r\\
\frac{\pi^2}{6}-\frac{1}{1^2}-\ldots-\frac{1}{(2r)^2}=\frac{\pi^2}{6}-\sum\limits_{1\leq k\leq 2r} \frac{1}{k^2} \text{ if } n=r
\end{cases}
\]

Then for every $r$ we have

\begin{equation}
\label{hrseries}
\sum_{n=0}^{+\infty} c_n h_r(n)=\frac{1}{r+1}
\end{equation}
\end{theorem}
\begin{remark}
Once again, by the convention for empty sums, we have $h_0(0)=\frac{\pi^2}{6}$. Particular cases of \eqref{hrseries} $r=0$ and $r=1$ can be rewritten in a following forms:
\begin{equation}
    \sum_{n=1}^{+\infty}c_n\frac{H_{2n}-H_n}{n}=1-\frac{\pi^2}{6}
\end{equation}

for $r=0$ and
\begin{equation}
    \sum_{n=2}^{+\infty}c_n\frac{H_{2n}-H_{n+1}}{n-1}=\frac{\pi^2}{4}-\frac{19}{8}
\end{equation}

for $r=1$.

The formula for $h_r(r)$ can be considered as a "limit case" of the formula for $n\neq r$, because $h_r(r)=\lim\limits_{n\to r} h_r(n)$ if we allow $n$ to be an arbitrary real number.
\end{remark}

To prove Theorem \ref{harmonicrelation} we need one more equivalent form of the recursive formula for $c_n$.

\begin{lemma}
\label{funceq}
For any complex $|t|<1$ and nonnegative integer $n$ we define
\[
G_n(t)=\frac{1}{n+1}+\frac{t}{n+2}+\frac{t^2}{n+3}+\ldots=\sum_{k=0}^{+\infty}\frac{t^k}{n+k+1}.
\]
Then for any $t$ inside the unit disc we have
\begin{equation}
\label{pfunctional}
\sum_{n=0}^{+\infty} c_nt^n G_{2n}(t)=1.
\end{equation}
\end{lemma}
\begin{proof}[of lemma \ref{funceq}]
We will deduce the identity \eqref{pfunctional} from the integral equation \eqref{integraleq}. This is not the only possible proof, but it seems to be one of the simplest and also the most instructive one.

First of all, let us notice that the functions $G_n(t)$ look similar to the function $-\ln(1-t)$. For example, when $n=0$ we have $G_0(t)=-\ln(1-t)/t$. More precisely, we have

\[
t^{n+1}G_n(t)=\sum_{k=0}^{+\infty}\frac{t^{n+1+k}}{n+k+1}=\sum_{k\geq n+1}\frac{t^k}{k}.
\]
This last expression is the Taylor series for $-\ln(1-t)$, but without $n$ initial terms. Therefore,

\[t^{n+1}G_n(t)=-(\ln(1-t)+t+t^2/2+\ldots+t^n/n).
\]
It is also useful to note that $\frac{t^k}{k}=\int\limits_0^t y^{k-1}dy$. From this observation we obtain the formula

\[t^{n+1}G_n(t)=\int_0^t \sum_{k\geq n}y^k dy.
\]

The sum inside the integral can be evaluated via the geometric series and we get

\[t^{n+1}G_n(t)=\int_0^t \frac{y^k}{1-y}dy.
\]
Finally, from the change of variables $y=xt$ we see that

\[G_n(t)=\int_0^1 \frac{x^k}{1-tx}dx.
\]
Now, plugging this into the left-hand side of the identity \eqref{pfunctional} we get

\[\sum_{n=0}^{+\infty} c_nt^nG_{2n}(t)=\sum_{n=0}^{+\infty}c_n\int_0^1 \frac{t^n x^{2n}}{1-tx}dx=\int_0^1 \frac{F(tx^2)}{1-tx}dx,
\]

where $F(x)=\sum\limits_{n=0}^{+\infty} c_nx^n$. This last expression is equal to $1$ by the integral equation \eqref{integraleq}, which concludes the proof.
\end{proof}

Theorem \ref{harmonicrelation} follows quite easily from the Lemma \ref{funceq}.
\begin{proof}[of Theorem \ref{harmonicrelation}]
Let us multiply both sides of the formula \eqref{pfunctional} by $t^r$ and integrate over the interval $[0,1]$. We get the following identity:

\[
\sum_{n=0}^1 c_n\int_0^{+\infty}t^{n+r}G_{2n}(t)dt=\int_0^1 t^rdt=\frac{1}{r+1}.
\]

Therefore, to prove our formula it suffices to show that $\int\limits_0^1 t^{n+r}G_{2n}(t)dt=h_r(n)$.

This fact can be deduced by direct calculation from the infinite series expansion of $G_{2n}(t)$ as follows:

\[\int_0^1 t^{n+r}G_{2n}(t)dt=\sum_{k=0}^{+\infty}\int_0^1 \frac{t^{k+n+r}}{2n+k+1}=\sum_{k\geq 0}\frac{1}{(k+n+r+1)(2n+k+1)}.
\]

There are two possible cases: $n=r$ or $n\neq r$. In the first case we have

\[\int_0^1 t^{n+r}G_{2n}(t)dt=\sum_{k\geq 0}\frac{1}{(k+2r+1)^2}=\sum_{k\geq 2r+1}\frac{1}{k^2}=\frac{\pi^2}{6}-\sum_{0<k\leq 2r}\frac{1}{k^2}=h_r(r),
\]

as needed. Now, when $n\neq r$ we have

\[\frac{1}{(k+n+r+1)(2n+k+1)}=\frac{1}{n-r}\left(\frac{1}{k+n+r+1}-\frac{1}{2n+k+1}\right).
\]

From this formula we get

\[
\int_0^1 t^{n+r}G_{2n}(t)dt=\frac{1}{n-r}\lim_{N\to +\infty}\sum_{k=0}^N \left(\frac{1}{k+n+r+1}-\frac{1}{2n+k+1}\right)=
\]
\[\frac{1}{n-r}\lim_{N\to +\infty}\left(\sum_{k=n+r+1}^{N+n+r+1} \frac{1}{k}-\sum_{k=2n+1}^{2n+N+1}\frac{1}{k}\right)=\]
\[=\frac{1}{n-r}\lim_{N\to +\infty}(H_{N+n+r+1}-H_{n+r}-H_{2n+N+1}+H_{2n})=\frac{H_{2n}-H_{n+r}}{n-r},
\]

as needed. The last equality holds because $|H_{N+n+r+1}-H_{2n+N+1}|=O_{n,r}(N^{-1})$ and hence the limit is equal to 0. Due to previous considerations, this finishes our proof.
\end{proof}

\section{Acknowledgements}

We thank M.\,A.\,Korolev and A.\,V.\,Ustinov for discussion and suggestions. The first author is partially supported by Laboratory of Mirror Symmetry NRU HSE, RF Government grant, ag. \textnumero14.641.31.0001, the Simons Foundation, the "Young Russian Mathematics" contest and by the Russian Science Foundation under grant \textnumero18-41-05001.

\end{fulltext}
\end{document}